\documentclass[11pt]{amsproc}

\usepackage{mathtext}
\usepackage[cp1251]{inputenc}

\usepackage{bm}

\usepackage[dvips]{graphicx}
\usepackage{amsmath}
\usepackage{amssymb}
\usepackage{amsxtra}

\usepackage{epsfig}
\usepackage{epic}
\usepackage{eepic}
\usepackage{graphics}
\usepackage{graphicx}
\usepackage{subfigure}

\usepackage{caption}
\def\N{{{\Bbb N}}}
\def\Z{{{\Bbb Z}}}
\def\T{{{\Bbb T}}}
\def\R{{\Bbb R}}

\def\l{{\lambda }}
\def\a{{\alpha }}
\def\D{{\Delta }}

\def\a{{\alpha}}

\def\d{{\delta}}
\def\e{{\varepsilon}}
\def\s{{\sigma}}
\def\vp{{\varphi}}

\def\g{{\gamma }}

\def\){\right)}
\def\({\left(}
\def\supp{\operatorname{supp}}

\numberwithin{equation}{section}

\newtheorem{lemma}{Lemma}[section]
\newtheorem{theorem}{Theorem}[section]

\newtheorem{remark}{Remark}[section]

\par

\begin{document}

\title[]{On generalized $K$-functionals in $L_p$ for $0<p<1$}

\author[Yurii
Kolomoitsev]{Yurii
Kolomoitsev$^{\text{a, 1, *}}$}
\address{Institute for Numerical and Applied Mathematics, G\"ottingen University, Lotzestr. 16-18, 37083 G\"ottingen, Germany}
\email{kolomoitsev@math.uni-goettingen.de}

\author[Tetiana
Lomako]{Tetiana
Lomako$^{\text{a}}$}
\address{Institute for Numerical and Applied Mathematics, G\"ottingen University, Lotzestr. 16-18, 37083 G\"ottingen, Germany}
\email{tetiana.lomako@uni-goettingen.de}

\thanks{$^\text{a}$Institute for Numerical and Applied Mathematics, G\"ottingen University, Lotzestr. 16-18, 37083 G\"ottingen, Germany}

\thanks{$^1$Support by the German Research Foundation in the framework of the RTG 2088}

\thanks{$^*$Corresponding author}

\thanks{E-mail address: kolomoitsev@math.uni-goettingen.de}

\date{\today}
\subjclass[2010]{26A33, 46E35, 42A10, 42A45, 41A30} \keywords{$K$-functional, $L_p$ with $0<p<1$, homogeneous multipliers, fractional derivatives, quadrature formula}

\begin{abstract}
We show that the Peetre $K$-functional between the space $L_p$ with $0<p<1$ and the corresponding smooth function space $W_p^\psi$ generated by the Weyl-type differential operator $\psi(D)$, where $\psi$ is a homogeneous function of any positive order, is identically zero. The proof of the main results is based on the properties of the de la Vall\'ee Poussin kernels and the quadrature formulas for trigonometric polynomials and entire functions of exponential type.
\end{abstract}

\maketitle

\section{Introduction}

The classical Peetre $K$-functional is defined by
$$
K(f,t;X,Y):=\inf_{g\in Y}(\Vert f-g\Vert_X+t|g|_Y),
$$
where $(X,\Vert \cdot\Vert_X)$ is a (quasi)-Banach space and $Y\subset X$ is a complete subspace with semi-norm $|\cdot|_Y$.
The $K$-functional is one of the main tool in the theory of interpolation spaces. Moreover, it has important applications in approximation theory. Namely, smoothness properties of a function as well as errors of various approximation methods can be efficiently expressed by means of $K$-functionals, especially when the classical moduli of smoothness cannot be applied, see, e.g.,~\cite{Di98}, \cite{Dr10}, \cite{K12}, \cite{KT20_2}, \cite{KT19m}.

In this paper, we are interested in the case, where $X$ is an $L_p$ space and $Y$ is a smooth function space $W_p^\psi$ generated by the Weil-type differential operator $\psi(D)$, where $\psi$ is a homogeneous function. The class of such differential operators includes, for example, the classical partial derivatives, Weyl and Riesz derivatives, the Laplace-operator and its (fractional) powers.
Let us consider the $K$-functional for the pair $(L_p(\T), W_p^\a(\T))$, where $\T$ is the circle and $W_p^\a(\T)$  is the fractional Sobolev space defined via the Weyl derivative of order $\a>0$, i.e.,
\begin{equation}\label{KK}
  K(f,\d^r;L_p,W_p^\a)=\inf_{g\in W_p^\a(\T)}(\|f-g\|_{L_p(\T)}+\d^r\|g^{(\a)}\|_{L_p(\T)}).
\end{equation}
It is well-known that if $1\le p\le \infty$, then this $K$-functional is equivalent to the classical modulus of smoothness of order $\a$, see~\cite{johnen} for the case $\a\in\N$ and~\cite{BDGS77} for arbitrary $\a>0$. A similar result for the Riesz derivative and special modulus of smoothness was established in~\cite{RS3}. Properties and applications of the $K$-functionals between the space $L_p$ on the torus $\T^d$ or $\R^d$ and the corresponding smooth function space $W_p^\psi$ with a particular homogeneous function $\psi$ were studied in~\cite{ARS19}, \cite{Di98}, \cite{KT20}, \cite{run}, \cite{Wil}. Also, there are many works dedicated to the study of $K$-functionals in different quasi-normed Hardy spaces $H_p$, $0<p<1$, see, e.g., \cite{K11}, \cite{K12}, \cite{KT19m}, \cite[Ch.~4]{Lu95}. In particular, as in the case of the Banach spaces $L_p$, the $K$-functional of type~\eqref{KK} in the quasi-normed Hardy spaces is equivalent to the corresponding modulus of smoothness of integer or fractional order, see, e.g., \cite{K11}, \cite{K12}, \cite[Ch.~4]{Lu95}.

In contrast to the case of Banach spaces and quasi-normed Hardy spaces,
the $K$-functionals in $L_p$ with $0<p<1$ are no longer relevant. Namely,
it was shown in~\cite{DHI} that the $K$-functional~\eqref{KK} with $0<p<1$ and the derivative of integer order $\a\in \N$ is identically zero.  In~\cite{run}, exploiting the approach from~\cite{DHI}, the same property was established for the $K$-functional between the space $L_p(\T^d)$ and the smooth function space $W_p^\psi(\T^d)$, where $\psi$ is a homogeneous function of order $\a\ge 1$ if $d=1$ and $\a\ge 2$ if $d\ge 2$. Note that the restriction on the parameter $\a$ is due to the fact that the proof of the above property in~\cite{run} is essentially based on the results in~\cite{DHI} obtained for the derivatives of integer orders.
But it is well known that a solution of problems involving fractional smoothness in $L_p$ with $0<p<1$ usually is more involved than its integer counterparts and requires very often development essentially new approaches, see, e.g.,~\cite{BL}, \cite{RS}, \cite{K17}, \cite{KL19}.

In the papers~\cite{KT20} and~\cite{KT19m}, it was stated without the proof that the $K$-functional  $K(f,t;L_p(\Omega),W_p^\a(\Omega))$, where $\Omega=\T^d$ or $\R^d$, is identically zero for any positive $\a>0$ and $0<p<1$. But, as it was pointed by S. Artamonov, this fact has not yet been established anywhere. The purpose of the present paper is to improve this drawback by showing that in the case $0<p<1$, the $K$-functional is identically zero for various differential operators $\psi(D)$ generated by a homogeneous function $\psi$ of any order $\a>0$. Our approach is different from one presented in~\cite{DHI} and~\cite{run} and is based on properties of the de la Vall\'ee Poussin kernels and the quadrature formulas for trigonometric polynomials and entire functions of exponential type.

\bigskip

\section{Notations and definitions}

Let $\R^d$  be the $d$-dimensional Euclidean space with elements $x=(x_1,\dots,x_d)$, and $(x,y)=x_1y_1+\dots+x_dy_d$,
$|x|=(x,x)^{1/2}$. Let $\N$ be the set of positive integers, $\Z^d$ be the integer lattice in $\R^d$, and $\T^d=\R^d / 2\pi \Z^d$. By $\{e_j\}_{j=1}^d$ we denote the standard basis in $\R^d$.
For $n\in\N$, the space of trigonometric polynomials of degree at most than $n$ is defined by
$$
\mathcal{T}_n={\rm span}\{e^{i(k,x)}: k\in [-n,n]^d\}.
$$

In what follows, $\Omega=\T^d$ or $\R^d$. As usual, the space $L_p(\Omega)$ consists of all measurable functions $f$ such that
for $0<p<\infty$
$$
\Vert f\Vert_{L_p(\Omega)}=\bigg(\int_{\Omega}|f(x)|^p
dx\bigg)^\frac1p<\infty.
$$
Note that  $\Vert f\Vert_{L_p(\Omega)}$ for $0<p<1$ is a quasi-norm satisfying $\Vert f+g\Vert^p_{L_p(\Omega)}\le \Vert f\Vert^p_{L_p(\Omega)}+\Vert g\Vert^p_{L_p(\Omega)}$.
By $C_0(\R^d)$, we denote the set of all continuous functions $f$ such that $\lim_{|x|\to\infty}f(x)=0$. For any $q\in (0,\infty]$, we set
$$
q_1=\left\{
      \begin{array}{ll}
        q, & \hbox{$0<q<1$,} \\
        1, & \hbox{$1\le q\le \infty$.}
      \end{array}
    \right.
$$
If $f\in L_1(\T^d)$, then its $k$-th Fourier coefficient is defined by
$$
\widehat{f}(k)=\frac1{(2\pi)^{d}}\int_{\T^d} f(x)e^{-i(x,k)}dx.
$$
By $\D_h^r f$, where $r\in\N$ and $h\in \R^d$, we denote the symmetric difference of the function $f$,
$$
\D_h^r f(x)=\sum_{\nu=0}^r (-1)^\nu \binom{r}{\nu}f(x-(\tfrac r2-\nu)h).
$$

We say that a function $\psi$ belongs to the class $\mathcal{H}_\a$, $\a\in\R$,
if  $\psi(\xi)\neq 0$ for $\xi\in \R^d\setminus \{0\}$, $\psi \in C^\infty(\R^d\setminus \{0\})$, and
$\psi$ is  a homogeneous function of order $\a$, i.e.,
$$
\psi(\tau \xi)=\tau^\a \psi(\xi),\quad \tau>0,\quad \xi\in\R^d.
$$
Any function $\psi\in \mathcal{H}_\a$ generates the Weyl-type differentiation operator as follows:
$$
\psi(D)\,:\,
\sum_{k\in \Z^d}  c_k e^{i(k,x)}\to {\sum_{k\in \Z^d\setminus\{0\}}} \psi(k) c_k e^{i(k,x)}.
$$
Important  examples of the Weyl-type  operators are the following:

\begin{enumerate}
  \item[-]\; the linear differential operator
$$
P_m({D})f=
\sum_{{}_{\quad k\in \Z_+^d}^{ k_1+\cdots+k_d=m}}a_k {D}^k f,\qquad
{D}^k=\frac{\partial^{k_1+\cdots+k_d}}{\partial x_1^{k_1}
\cdots \partial x_d^{k_d}},
$$
$\phantom{ahhhn}$with
$$
\psi(\xi)=\sum_{{}_{\quad k\in \Z_+^d}^{ k_1+\cdots+k_d=m}}a_k (i\xi_1)^{k_1}\dots(i\xi_d)^{k_d};
$$
  \item [-]\; the fractional Laplacian  $(-\Delta)^{\alpha/2}f$ with $\psi(\xi)=|\xi|^\alpha$, $\xi\in\R^d$;
  \item [-] \;
  the classical
Weyl derivative $f^{(\alpha)}$ with $\psi(\xi)=(i \xi)^\alpha$, $\xi\in\R$.
\end{enumerate}

Let $\psi\in \mathcal{H}_\a$, $\a>0$ and $0<p\le 1$. By ${W_p^\psi}(\T^d)$ we denote the space of $\psi$-smooth functions in $L_p(\T^d)$, i.e.,
$$
W_p^\psi(\T^d)=\left\{g\in L_1(\T^d)\,:\, \psi(D)g\in L_p(\T^d)\right\}
$$
with
$$
|g|_{W_p^\psi}=\|\psi(D)g\|_{L_p(\T^d)}.
$$

\bigskip

\section{Main result in the periodic case}

\begin{theorem}\label{th1}
  Let $0<p<1$, $0<q\le\infty$, $\a>\max\{0,d(1-\frac1q)\}$, and $\psi\in \mathcal{H}_\a$. Then, for any $f\in L_p(\T^d)$ and $\d>0$, we have
  \begin{equation*}
    K\big(f,\d,L_q(\T^d),W_p^\psi(\T^d)\big)=0.
  \end{equation*}
\end{theorem}

To prove this theorem, we need the following auxiliary results and notations.
In what follows, the de la Vall\'{e}e Poussin type kernel is defined by
\begin{equation*}
    V_n(x):=\sum_{k\in\Z^d} v\left(\frac kn\right)e^{i(k,x)},
\end{equation*}
where $v\in C^\infty(\R^d)$,
$v(\xi)=1$ for $\xi\in [-1,1]^d$ and $v(\xi)=0$ for $\xi\in \R^d\setminus [-2,2]^d$.

\begin{lemma}\label{leval} {\sc (See~\cite[Ch. 4 and Ch. 9]{TB}.)}
Let $0<p\le 1$ and $\vp\in C^\infty(\R^d)$ have a compact support. Then
$$
\sup_{\e>0} \e^{d(1-\frac1p)}\bigg\Vert
\sum_{k\in \Z^d}\varphi(\e k)  e^{i(k,x)}\bigg\Vert_{L_p(\T^d)}<\infty.
$$
In particular, $\|V_n\|_{L_p(\T^d)}\le c_p n^{d(1-\frac1p)}$.
\end{lemma}

We will also use the following quadrature formula and the Marcinkiewicz-Zygmund inequality.

\begin{lemma}\label{leqmz}
  Let $T_n\in \mathcal{T}_n$, $t_{k,n}=\frac{2\pi k}{2n+1}$, $k\in [0,2n]^d$, and $0<p<\infty$. Then
  \begin{equation}\label{qu}
    \frac1{(2\pi)^d}\int_{\T^d} T_n(x)dx=\frac1{(2n+1)^d}\sum_{k\in [0,2n]^d}T_n\(t_{k,n}\)
  \end{equation}
and
\begin{equation}\label{mz}
    \frac1{(2n+1)^d}\sum_{k\in [0,2n]^d}\left|T_n\(t_{k,n}\)\right|^p\le C_p\|T_n\|_{L_p(\T^d)}^p.
  \end{equation}
\end{lemma}

\begin{proof}
Equality~\eqref{qu} can be obtained by applying the univariate quadrature formulas for trigonometric polynomials in~\cite[Ch.~X, (2.5)]{Z} to each variable one after another. Similarly, using the univariate Marcinkiewicz-Zygmund inequality in~\cite[Theorem~2]{LMN}, we can prove~\eqref{mz}.
\end{proof}

\begin{proof}[Proof of Theorem~\ref{th1}]
In what follows, for simplicity, we write $\Vert f\Vert_p=\Vert f\Vert_{L_p(\T^d)}$.
Let $\e>0$ be fixed and let $T_\mu\in \mathcal{T}_\mu$ be such that
\begin{equation*}
  \|f-T_\mu\|_q^{q_1}<\frac\e3.
\end{equation*}
It is clear that
\begin{equation}\label{t1}
  K\big(f,1,L_q(\T^d),W_p^\psi(\T^d)\big)^{q_1}<\frac\e3+K\big(T_\mu,1,L_q(\T^d),W_p^\psi(\T^d)\big)^{q_1}.
\end{equation}
Let $m>\mu$, $m\in \N$. We set
\begin{equation*}
  \mathcal{V}_{2^m}(x)=-\frac1{4}\sum_{j=1}^d \D_{e_j}^2V_{2^m}(x)=\sum_{k\in\Z^d}(\sin^2 k_1+\dots+\sin^2 k_d )v\(\frac{k}{2^m}\)e^{i(k,x)}
\end{equation*}
and
\begin{equation*}
  \psi_1(\xi)=\left\{
                \begin{array}{ll}
                  \displaystyle 0, & \hbox{$\xi=0$,} \\
                  \displaystyle\frac{\psi(\xi)}{\sin^2 \xi_1+\dots+\sin^2 \xi_d}, & \hbox{$\xi\in \R^d\setminus \{0\}$.}
                \end{array}
              \right.
\end{equation*}
Then, denoting $\tilde{\psi}(\xi)=\tfrac1{\psi(\xi)}$, we see that equality~\eqref{qu} implies
\begin{equation}\label{t2}
\begin{split}
   T_\mu(x)&=\frac1{(2\pi)^d}\int_{\T^d}\psi_1(D)T_\mu(t)\cdot {\tilde\psi}(D)\mathcal{V}_{2^{m}}(x-t)dt+\widehat{T_\mu}(0)\\
   &=\frac1{(2M+1)^d}\sum_{\ell\in [0,2M]^d}\psi_1(D)T_\mu(t_\ell)\cdot {\tilde\psi}(D)\mathcal{V}_{2^{m}}(x-t_\ell)+\widehat{T_\mu}(0),
\end{split}
\end{equation}
where $M=\mu+2^{m+1}$ and $t_\ell=t_{\ell,M}=\frac{2\pi \ell}{2M+1}$.

Let $n>m$, $n\in \N$. From the definition of the $K$-functional, it follows that
\begin{equation}\label{t3}
  \begin{split}
     K&\big(T_\mu,1,L_q(\T^d),W_p^\psi(\T^d)\big)\\
&\le \bigg\|T_\mu-\frac1{(2M+1)^d}\sum_{\ell\in [0,2M]^d}\psi_1(D)T_\mu(t_\ell)\cdot {\tilde\psi}(D)\mathcal{V}_{2^{n}}(x-t_\ell)-\widehat{T_\mu}(0)\bigg\|_q\\
  &+\bigg\|\frac1{(2M+1)^d}\sum_{\ell\in [0,2M]^d}\psi_1(D)T_\mu(t_\ell)\cdot \mathcal{V}_{2^{n}}(x-t_\ell)\bigg\|_p=I_1+I_2.
  \end{split}
\end{equation}
Using~\eqref{t2}, \eqref{mz}, and  a telescopic sum, we obtain
\begin{equation}\label{t4}
  \begin{split}
     I_1^{q_1}&=\bigg\|\frac1{(2M+1)^d}\sum_{\ell\in [0,2M]^d}\psi_1(D)T_\mu(t_\ell)\cdot {\tilde\psi}(D)\big(\mathcal{V}_{2^{n}}(x-t_\ell)-\mathcal{V}_{2^{m}}(x-t_\ell)\big)\bigg\|_q^{q_1}\\
&\le \frac1{(2M+1)^{dq_1}}\sum_{\ell\in [0,2M]^d}|\psi_1(D)T_\mu(t_\ell)|^{q_1}
\big\|{\tilde\psi}(D)\big(\mathcal{V}_{2^{n}}(x-t_\ell)-\mathcal{V}_{2^{m}}(x-t_\ell)\big)\big\|_q^{q_1}\\
&\le C_{q_1} (2M+1)^{d(1-q_1)}\|\psi_1(D)T_\mu\|_{q_1}^{q_1}\big\|{\tilde\psi}(D)\(\mathcal{V}_{2^{n}}-\mathcal{V}_{2^{m}}\)\big\|_q^{q_1}\\
&\le C_{q_1} (2M+1)^{d(1-q_1)}\|\psi_1(D)T_\mu\|_{q_1}^{q_1}\sum_{\nu=m}^{n-1}
\big\|{\tilde\psi}(D)\(\mathcal{V}_{2^{\nu+1}}-\mathcal{V}_{2^{\nu}}\)\big\|_q^{q_1}.
  \end{split}
\end{equation}
Next, denoting
$$
\mathcal{N}_{2^\nu}(x)=\sum_{k\in\Z^d}\eta\(\frac{k}{2^\nu}\)e^{i(k,x)}\quad\text{with}\quad \eta(\xi)=\frac{v(\tfrac\xi2)-v(\xi)}{\psi(\xi)},
$$
we get
\begin{equation*}
  \begin{split}
     {\tilde\psi}(D)\(\mathcal{V}_{2^{\nu+1}}(x)-\mathcal{V}_{2^{\nu}}(x)\)&=\frac1{2^{\a\nu}}\sum_{k\in\Z^d}(\sin^2 k_1+\dots+\sin^2 k_d ) \eta\(\frac{k}{2^\nu}\)e^{i(k,x)}\\
&=-\frac1{2^{\a\nu+2}}\sum_{j=1}^d\D_{e_j}^2\mathcal{N}_{2^\nu}(x)
  \end{split}
\end{equation*}
and hence
\begin{equation}\label{t7}
  \begin{split}
     \big\|{\tilde\psi}(D)\(\mathcal{V}_{2^{\nu+1}}-\mathcal{V}_{2^{\nu}}\)\big\|_q^{q_1}&=
\frac1{2^{(\a\nu+2)q_1}}\bigg\|\sum_{j=1}^d\D_{e_j}^2\mathcal{N}_{2^\nu}\bigg\|_q^{q_1}\le \frac{4^{1-q_1}d}{2^{\a\nu q_1}}\big\|\mathcal{N}_{2^\nu}\big\|_q^{q_1}.
  \end{split}
\end{equation}
If $0<q\le 1$, then with the help of Lemma~\ref{leval}, we obtain
\begin{equation}\label{zv1}
  \big\|\mathcal{N}_{2^\nu}\big\|_q^{q}\le \frac{c_{q,\eta}}{2^{(\a q+d(1-q))\nu}},
\end{equation}
where we have usee the fact that $\eta\in C^\infty(\R^d)$ and $\supp\eta$ is compact. Next, for $1<q\le \infty$, exploiting the Nikolskii inequality of different metrics (see, e.g.,~\cite[4.3.6]{TB}) and again Lemma~\ref{leval}, we get
\begin{equation}\label{zv2}
\big\|\mathcal{N}_{2^\nu}\big\|_q\le 2^{d(1-\frac1q)\nu} \big\|\mathcal{N}_{2^\nu}\big\|_1\le  \frac{c_{1,\eta}}{2^{(\a +d(\frac1q-1))\nu}}.
\end{equation}
Thus, inequalities~\eqref{t4}, \eqref{t7}, \eqref{zv1}, \eqref{zv2}, and the condition $\a>\max\{0,d(1-\frac1q)\}$ implies that, for sufficiently large $m>m_0(T_\mu,\psi,q,\e)$,
\begin{equation}\label{t8}
  \begin{split}
     I_1^{q_1}&\le \frac{2^{q_1(\a+d(\frac1q-1))+2(1-q_1)}dC_{q_1}c_{{q_1},\eta}}{2^{q_1(\a+d(\frac1q-1))}-1} \|\psi_1(D)T_\mu\|_{q_1}^{q_1}\cdot\frac{(2^{m+2}+2\mu+1)^{d(1-{q_1})}}{2^{q_1(\a+d(\frac1q-1))m}}<\frac\e3.
  \end{split}
\end{equation}

Next, we estimate $I_2$. Application of the Marcinkiewicz-Zygmund inequality~\eqref{mz} and Lemma~\ref{leval} yields
\begin{equation}\label{t9}
  \begin{split}
     I_2^p&\le \frac1{(2M+1)^{dp}}\sum_{\ell\in [0,2M]^d}|\psi_1(D)T_\mu(t_\ell)|^p\|\mathcal{V}_{2^{n}}\|_p^p\\
&\le 4dC_p(2M+1)^{d(1-p)}\|\psi_1(D)T_\mu\|_p^p\|V_{2^{n}}\|_p^p\\
&\le 4dc_pC_p(2M+1)^{d(1-p)}\|\psi_1(D)T_\mu\|_p^p\cdot2^{d(p-1)n}<\(\frac\e{3}\)^{p/q_1}
  \end{split}
\end{equation}
for sufficiently large $n>n_0(T_\mu,m,\psi,p,\e)$.

Finally, combining~\eqref{t1}, \eqref{t3}, \eqref{t8}, and \eqref{t9} with appropriate $n>n_0$ and $m>m_0$, we obtain that $K\big(f,1,L_q(\T^d),W_p^\psi(\T^d)\big)^{q_1}<\e$. This proves the theorem.
\end{proof}

\bigskip

\section{Main result in the non-periodic case}
To formulate an analogue of Theorem~\ref{th1} for non-periodic functions, we introduce addtional notations.
As usual, by $\mathcal{S}$ and $\mathcal{S}'$ we denote the Schwartz space of infinitely differentiable rapidly decreasing  functions on $\R^d$ and its dual (the space of tempered distributions), respectively. The Fourier transform of $f\in L_1(\R^d)$ is given by
$$
\mathcal{F}f(\xi)=\widehat{f}(\xi)=\frac1{(2\pi)^{d/2}}\int_{\R^d} f(x)e^{-i(x,\xi)}dx.
$$
As usual,
$$
f*g(x)=\int_{\R^d}f(y)g(x-y)dx
$$
is the convolution of two appropriate functions $f$ and $g$.
For $f\in \mathcal{S}'$, we define the Fourier transform by $\langle\widehat{f},\widehat{\vp}\rangle=\langle f, g\rangle$, $\vp\in \mathcal{S}$. Next, by  $\mathcal{B}_{\s,p}=\mathcal{B}_{\s,p}(\R^d)$, $\s>0$, $0<p\le \infty$, we denote the Bernstein space of entire functions of exponential type~$\s$. That is, $f \in \mathcal{B}_{\s,p}$ if $f\in L_p(\R^d)\cap \mathcal{S}'(\R^d)$ and $\supp \mathcal{F}f\subset [-\s,\s]^d$.
Recall (see, e.g.~\cite[Ch.~3]{nikol-book} and~\cite[Sec.~1]{TribF}) that if $0<p<q\le \infty$, then the following continuous embedding holds
\begin{equation*}
  \mathcal{B}_{\s,p} \hookrightarrow \mathcal{B}_{\s,q}.
\end{equation*}

Let $\psi\in \mathcal{H}_\a$, $\a>0$ and $0<p\le 1$. Similarly as in the periodic case, by ${W_p^\psi}(\R^d)$ we denote the space of $\psi$-smooth functions in $L_p(\R^d)$, that is,
$$
W_p^\psi(\R^d)=\left\{g\in \mathcal{S}(\R^d)\,:\, \psi(D)g\in L_p(\R^d)\right\}
$$
with
$$
|g|_{W_p^\psi(\R^d)}=\|\psi(D)g\|_{L_p(\R^d)},\quad
\text{where}\quad
\psi(D)g=\mathcal{F}^{-1}(\psi \widehat{g}).
$$

\begin{theorem}\label{th2}
  Let $0<p<1$, $0<q\le \infty$, $\a>\max\{d(\frac1p-1), d(1-\frac1q)\}$, and $\psi\in \mathcal{H}_\a$. Then, for any $f\in L_p(\R^d)$ ($f\in C_0(\R^d)$ if $q=\infty$) and $\d>0$, we have
  \begin{equation}\label{t0+}
    K\big(f,\d,L_q(\R^d),W_p^\psi(\R^d)\big)=0.
  \end{equation}
\end{theorem}

\bigskip

To prove Theorem~\ref{th2}, we will need the following analogue of Lemma~\ref{leqmz} for entire functions of exponential type.

\begin{lemma}\label{leqmz+}
$1)$  Let $1\le p\le \infty$, $1/p+1/q=1$, and $\s>0$. Then, for all $g\in \mathcal{B}_{\pi\s,p}$ and $h\in \mathcal{B}_{\pi\s,q}$, we have
  \begin{equation}\label{qu+}
    (g*h)(x)=\frac1{\s^d}\sum_{k\in\Z^d}g\(\frac k\s\)h\(x-\frac k\s\),\quad x\in \R^d.
  \end{equation}
The series on the right-hand side of~\eqref{qu+} converges absolutely for all $x\in\R^d$ and this converges is uniform on each compact subset of $\R^d$.

$2)$ Let $0<p<\infty$ and $\s>0$. Then, for all $g\in \mathcal{B}_{\pi\s,p}$, we have
\begin{equation}\label{mz+}
    \frac1{\s^d}\sum_{k\in\Z^d}\left|g\(\frac k\s\)\right|^p\le C_p\|g\|_{L_p(\R^d)}^p.
  \end{equation}
\end{lemma}

Inequality~\eqref{qu+} can be found, e.g., in~\cite[Lemma~6.2]{stens}. For the Plancherel–Polya-type inequality~\eqref{mz+}, see, e.g.,~\cite[4.3.1]{TB}. Note that the one-dimensional versions of both inequalities were considered in~\cite{SS00}.

Recall also the following convolution inequality, see, e.g.,~\cite[1.5.3]{TribF}.

\begin{lemma}\label{leconv}
  Let $0<p\le 1$ and $\s>0$. Then, for all $f,g\in \mathcal{B}_{\s,p}$, we have
\begin{equation*}
  \|f*g\|_{L_p(\R^d)}\le c_{p}\s^{d(\frac1p-1)}\|f\|_{L_p(\R^d)}\|g\|_{L_p(\R^d)}.
\end{equation*}
\end{lemma}

\begin{proof}[Proof of Theorem~\ref{th2}]
The proof of the theorem is similar to the one of Theorem~\ref{th1}. However, because several steps are different, we present a detailed proof.

In what follows, we denote $\Vert \cdot\Vert_p=\Vert \cdot\Vert_{L_p(\R^d)}$. Let $\e>0$ be fixed and let $g_\mu\in \mathcal{S}$ such that $\supp \widehat{g_\mu} \subset [-2^\mu,2^\mu]^d$ and
\begin{equation*}\label{e3+}
  \|f-g_\mu\|_q^{q_1}<\frac\e3.
\end{equation*}
Then, as in the periodic case, we have
\begin{equation}\label{t1+}
  K\big(f,1,L_q(\R^d),W_p^\psi(\R^d)\big)^{q_1}<\frac\e3+K\big(g_\mu,1,L_q(\R^d),W_p^\psi(\R^d)\big)^{q_1}.
\end{equation}
For $\l>\mu$, $\l\in\N$, we introduce the functions
$$
f_{\mu,\l}(x)=\mathcal{F}^{-1}\(v(2^\l\xi)\widehat{g_\mu}(\xi)\)(x)
$$
and
$$
g_{\mu,\l}(x)=g_\mu(x)-f_{\mu,\l}(x).
$$
Next, let $m>\mu$, $m\in\N$, and
$$
V_{2^m}(x)=\mathcal{F}^{-1}(v_m(\xi))(x)\quad\text{with}\quad v_m(\xi)=\(2^{-m}{\xi}\).
$$
Further, we denote
$$
d_2=\left\{
      \begin{array}{ll}
        1, & \hbox{$d=1$,} \\
        2, & \hbox{$d\ge 2$}
      \end{array}
    \right.\quad\text{and}\quad r_2=\left\{
      \begin{array}{ll}
        \lceil\a\rceil, & \hbox{$d=1$,} \\
        2\left\lceil\frac{\a}{2}\right\rceil, & \hbox{$d\ge 2$,}
      \end{array}
    \right.
$$
where $\lceil\cdot\rceil$ is the ceil function, and consider the functions
\begin{equation*}
  \begin{split}
      \mathcal{V}_{2^m}(x)&=\frac{1}{(2i)^{r_2}}\sum_{j=1}^d \D_{2^{-\mu}e_j}^{r_2} V_{2^m}(x)\\
      &=\mathcal{F}^{-1}\big((\sin^{r_2} 2^{-\mu}\xi_1+\dots+\sin^{r_2} 2^{-\mu}\xi_d)v_m(\xi)\big)(x)
   \end{split}
\end{equation*}
and
\begin{equation*}
  \psi_2(\xi)=\left\{
                \begin{array}{ll}
                  \displaystyle 0, & \hbox{$\xi=0$,} \\
                  \displaystyle \frac{\psi(\xi)}{\sin^{r_2} 2^{-\mu}\xi_1+\dots+\sin^{r_2} 2^{-\mu}\xi_d}, & \hbox{$\xi\in \R^d\setminus \{0\}$.}
                \end{array}
              \right.
\end{equation*}
Let us show that there exists $q>1$ such that
\begin{equation}\label{q1}
\psi_2(D)g_{\mu,\l}=\mathcal{F}^{-1}\(\psi_2 \widehat{g_{\mu,\l}}\)\in L_q(\R^d)
\end{equation}
and
\begin{equation}\label{q2}
  \tilde\psi(D)\mathcal{V}_{2^m}(x)=\mathcal{F}^{-1}\(\frac{v_m}{\psi_2}\)\in L_{q'}(\R^d),
\end{equation}
where $1/q+1/q'=1$ and $\tilde\psi=\frac1{\psi}$. Indeed, relation~\eqref{q1} is obvious since $\psi_2 \widehat{g_{\mu,\l}}\in \mathcal{S}$.
To verify~\eqref{q2}, we represent the function $\frac{v_m}{\psi_2}$ in the following form
\begin{equation}\label{pred}
  \frac{v_m(\xi)}{\psi_2(\xi)}=\sum_{j=1}^d \frac{\xi_j^{r_2}}{\psi(\xi)}\cdot\(\frac{\sin 2^{-\mu}\xi_j}{\xi_j}\)^{r_2}v_m(\xi)=\sum_{j=1}^d h_j(\xi)\vp_j(\xi),
\end{equation}
where
$$
h_j(\xi)=\frac{\xi_j^{r_2}}{\psi(\xi)}\in C^\infty(\R^d\setminus \{0\})\quad\text{and}\quad \vp_j(\xi)=\(\frac{\sin 2^{-\mu}\xi_j}{\xi_j}\)^{r_2}v_m(\xi)\in \mathcal{S}.
$$
Since $h_j$ is a homogeneous function of order $r_2-\a\ge 0$, we have that $\widehat{h_j}$ belongs to $C^\infty(\R^d\setminus \{0\})$ and it is homogeneous of order
$-(r_2-\a+d)$, see, e.g.~\cite[Theorems~7.1.16 and~7.1.18]{Ho}.
Thus, applying the properties of convolution and choosing $\s_m$ such that $\supp v_m\subset \{|\xi|<\s_m\}$, we obtain
\begin{equation}\label{con}
  \begin{split}
    |\mathcal{F}({h_j\vp_j})(\xi)|&=|\langle\widehat{h_j},\widehat{\vp_j}(\xi-\cdot)\rangle|
    =\bigg|\int_{\R^d}\widehat{h_j}(y)\widehat{\vp_j}(\xi-y)dy\bigg|\\
    &\le  \int_{|\xi-y|\le \s_m}|\widehat{h_j}(y)\widehat{\vp_j}(\xi-y)|dy\\
   &\le\max_{|\xi-y|\le \s_m}|\widehat{h_j}(y)| \int_{\R^d}|\widehat{\vp_j}(y)|dy\\
    &\le c_m'\max_{|\xi-y|\le \s_m}|y|^{-(r_2-\a+d)}\le c_m''|\xi|^{-\g}\\
  \end{split}
\end{equation}
for $|\xi|>2\s_m$ and $\g=r_2-\a+d\ge d$. Moreover, since $h_j\vp_j\in L_1(\R^d)$, it follows from the standard properties of the Fourier transform, that $\mathcal{F}({h_j\vp_j})\in C_0(\R^d)$. Therefore, $\mathcal{F}({h_j\vp_j})\in L_s(\R^d)$ for all $s>1$. In particular,  we have that $\mathcal{F}^{-1}({h_j\vp_j})\in L_{q'}(\R^d)$, which together with equality~\eqref{pred} implies~\eqref{q2}.

Now, taking into account~\eqref{q1}--\eqref{q2}, we can apply Lemma~\ref{leqmz+}, which yields
\begin{equation}\label{t2+}
\begin{split}
   g_{\mu,\l}(x)&=g_\mu(x)-f_{\mu,\l}(x)\\
   &=\int_{\R^d}\psi_2(D)g_{\mu,\l}(y) {\tilde\psi}(D)\mathcal{V}_{2^{m}}(x-y)dy\\
   &={M^{-d}}\sum_{\ell\in \Z^d}\psi_2(D)g_{\mu,\l}(t_\ell) {\tilde\psi}(D)\mathcal{V}_{2^{m}}(x-t_\ell),
\end{split}
\end{equation}
where $M=\mu+2^{m+1}$ and $t_\ell=\frac\ell{M}$.

Let $n>m$, $n\in\N$. We have
\begin{equation}\label{t3+}
  \begin{split}
     K&\big(g_\mu,1,L_q(\R^d),W_p^\psi(\R^d)\big)\\
 &\le \bigg\|g_\mu-M^{-d}\sum_{\ell\in \Z^d}\psi_2(D)g_{\mu,\l}(t_\ell) {\tilde\psi}(D)\mathcal{V}_{2^{n}}(x-t_\ell)-f_{\mu,\l}\bigg\|_p\\
  &+\bigg\|M^{-d}\sum_{\ell\in \Z^d}\psi_2(D)g_{\mu,\l}(t_\ell) \mathcal{V}_{2^{n}}(x-t_\ell)+\psi(D)f_{\mu,\l}\bigg\|_p=I_1+I_2.
  \end{split}
\end{equation}
Taking into account~\eqref{t2+} and using the Plancherel–Polya-type inequality~\eqref{mz+}, we obtain
\begin{equation}\label{t4+}
  \begin{split}
     I_1^{q_1}&=\bigg\|M^{-d}\sum_{\ell\in \Z^d}\psi_2(D)g_{\mu,\l}(t_\ell)\cdot {\tilde\psi}(D)\big(\mathcal{V}_{2^{n}}(x-t_\ell)-\mathcal{V}_{2^{m}}(x-t_\ell)\big)\bigg\|_q^{q_1}\\
&\le M^{-d{q_1}}\sum_{\ell\in \Z^d}|\psi_2(D)g_{\mu,\l}(t_\ell)|^{q_1}
\big\|{\tilde\psi}(D)\big(\mathcal{V}_{2^{n}}(x-t_\ell)-\mathcal{V}_{2^{m}}(x-t_\ell)\big)\big\|_q^{q_1}\\
&\le C_{q_1} M^{d(1-{q_1})}\|\psi_2(D)g_{\mu,\l}\|_{q_1}^{q_1}\big\|{\tilde\psi}(D)\(\mathcal{V}_{2^{n}}-\mathcal{V}_{2^{m}}\)\big\|_q^{q_1}\\
&\le C_{q_1} M^{d(1-{q_1})}\|\psi_2(D)g_{\mu,\l}\|_{q_1}^{q_1}\sum_{\nu=m}^{n-1}
\big\|{\tilde\psi}(D)\(\mathcal{V}_{2^{\nu+1}}-\mathcal{V}_{2^{\nu}}\)\big\|_q^{q_1}.
  \end{split}
\end{equation}
Note that in the above relations $\psi_2(D)g_{\mu,\l}\in L_p(\R^d)$ because $\psi_2 \widehat{g_{\mu,\l}}\in \mathcal{S}$.
Next, denoting
$$
\mathcal{N}_{2^\nu}(x)=\mathcal{F}^{-1}\(\eta\(2^{-\nu}\xi\)\)(x)\quad\text{with}\quad \eta(\xi)=\frac{v(\tfrac\xi2)-v(\xi)}{\psi(\xi)},
$$
we get
\begin{equation*}
  \begin{split}
     {\tilde\psi}(D)&\(\mathcal{V}_{2^{\nu+1}}(x)-\mathcal{V}_{2^{\nu}}(x)\)\\
&=\frac1{2^{\a\nu}}\mathcal{F}^{-1}\((\sin^{r_2} 2^{-\mu}\xi_1+\dots+\sin^{r_2} 2^{-\mu}\xi_d ) \eta\(2^{-\nu}\xi\)\)(x)\\
&=\frac1{2^{\a\nu}(2i)^{r_2}}\sum_{j=1}^d\D_{2^{-\mu}e_j}^{r_2}\mathcal{N}_{2^\nu}(x).
  \end{split}
\end{equation*}
Thus, taking into account that $\eta\in C^\infty(\R^d)$ and $\supp\eta$ is compact, we obtain
\begin{equation}\label{t7+}
  \begin{split}
     \big\|{\tilde\psi}(D)&\(\mathcal{V}_{2^{\nu+1}}-\mathcal{V}_{2^{\nu}}\)\big\|_q^{q_1}
=
\frac1{2^{(\a\nu+r_2)q_1}}\bigg\|\sum_{j=1}^d\D_{2^{-\mu}e_j}^{r_2}\mathcal{N}_{2^\nu}\bigg\|_q^{q_1}
\\
&\le \frac{2^{(1-q_1)r_2}d}{2^{\a\nu q_1}}\big\|\mathcal{N}_{2^\nu}\big\|_q^{q_1}
=
\frac{2^{(1-q_1)r_2}d\|\widehat{\eta}\|_q^{q_1}}{2^{q_1(\a+d(\frac1q-1))\nu}}
=
\frac{c_{q,\eta}}{2^{q_1(\a+d(\frac1q-1))\nu}}.
  \end{split}
\end{equation}
Then, combining~\eqref{t4+} and~\eqref{t7+}, it is easy to see that
\begin{equation}\label{t8+}
  \begin{split}
     I_1^{q_1}&\le \frac{2^{q_1(\a+d(\frac1q-1))+2(1-q_1)}dC_{q_1}c_{{q_1},\eta}}{2^{q_1(\a+d(\frac1q-1))}-1} \|\psi_1(D)g_{\mu,\l}\|_{q_1}^{q_1}\cdot\frac{(2^{m+2}+2\mu+1)^{d(1-{q_1})}}{2^{q_1(\a+d(\frac1q-1))m}}<\frac\e3
  \end{split}
\end{equation}
for sufficiently large $m>m_0(g_{\mu,\l},\psi,q,\e)$.

Further we find
\begin{equation}\label{d1}
  \begin{split}
     I_2^p&\le\bigg\|M^{-d}\sum_{\ell\in \Z^d}\psi_1(D)g_{\mu,\l}(t_\ell) \mathcal{V}_{2^{n}}(x-t_\ell)\bigg\|_p^p+\|\psi(D)f_{\mu,\l}\|_p^p=J_{1}^p+J_{2}^p.
  \end{split}
\end{equation}
First we estimate $J_2$. Taking into account that $\mathcal{F}^{-1}(\psi v)\in \mathcal{B}_{2,p}(\R^d)$ for $\a>d(1/p-1)$, see, e.g.,~\cite{RS} (this can also be verified as~\eqref{con}) and applying Lemma~\ref{leconv}, we obtain
\begin{equation*}
  \begin{split}
     J_2&=2^{-\a\l}\big\|\mathcal{F}^{-1}\(\psi(2^\l\cdot)v(2^\l\cdot)\widehat{g_{\mu}}\)\big\|_p
    =2^{-\a\l}\|\mathcal{F}^{-1}(\psi(2^\l\cdot)v(2^\l\cdot))*g_{\mu}\|_p\\
&\le c_p2^{-\a\l+d(\frac1p-1)(\mu+1)}\|\mathcal{F}^{-1}(\psi(2^\l\cdot)v(2^\l\cdot))\|_p\|g_\mu\|_p\\
&= c_p2^{-\a\l+d(\frac1p-1)(\mu+1)+d(\frac1p-1)\l}\|\mathcal{F}^{-1}(\psi v)\|_p\|g_\mu\|_p.
  \end{split}
\end{equation*}
Thus,  for sufficiently large $\l>\l_0(g_\mu,\psi,p,\e)$, we get that
\begin{equation}\label{d3}
  J_2^p<\frac12\(\frac\e3\)^{p/q_1}.
\end{equation}
Next, the Plancherel–Polya-type inequality~\eqref{mz+} yields
\begin{equation}\label{t9+}
  \begin{split}
     J_1^p&\le M^{-dp}\sum_{\ell\in\Z^d}|\psi_2(D)g_{\mu,\l}(t_\ell)|^p\|\mathcal{V}_{2^{n}}\|_p^p\\
&\le 2^{r_2}dC_pM^{d(1-p)}\|\psi_2(D)g_{\mu,\l}\|_p^p\|V_{2^{n}}\|_p^p\\
&= 2^{r_2}dC_pc_p(\mu+2^{m+1})^{d(1-p)}\|\psi_2(D)g_{\mu,\l}\|_p^p\|\widehat{v}\|_p^p\cdot2^{d(p-1)n}<\frac12\(\frac\e3\)^{p/q_1}
  \end{split}
\end{equation}
for sufficiently large $n>n_0(g_{\mu,\l},m,\psi,p,\e)$, which together with~\eqref{d1}, \eqref{d3}, and~\eqref{t9+} implies that
\begin{equation}\label{d4}
  I_2^{q_1}<\frac\e3.
\end{equation}

Finally, combining~\eqref{t1+}, \eqref{t3+}, \eqref{t8+}, and \eqref{d4} with appropriate $\l>\l_0$, $n>n_0$, and $m>m_0$, we obtain that $K\big(f,1,L_q(\R^d),W_p^\psi(\R^d)\big)^{q_1}<\e$, which proves the theorem.
\end{proof}

\begin{remark}
If we suppose that $0<q\le 1$ in Theorem~\ref{th2}, then~\eqref{t0+} holds for any $\a>0$.
Indeed, according to~\cite[Theorem~5.1 and Corollary~2.2]{A07}, there exists $g_\mu\in \mathcal{S}$ such that $\supp \widehat{g_\mu} \subset [-2^\mu,2^\mu]^d\setminus [-1,1]^d$ and $\|f-g_\mu\|_q^q<\frac\e3.$ Thus, in the proof of Theorem~\ref{th2}, we can put $\l=1$, $g_\mu=g_{\mu,1}$, and $f_{\mu,\l}=0$, which implies that we do not need to estimate the term $J_2$, where the restriction $\a>d(1/p-1)$ appears.
\end{remark}


\begin{thebibliography}{16}
{


\bibitem{A07} A. B. Aleksandrov,
Spectral subspaces of the space $L_p$, $p<1$. (Russian)
Algebra i Analiz \textbf{19} (2007), no.~3, 1--75; translation in
St. Petersburg Math. J. \textbf{19} (2008), no.~3, 327--374.

\bibitem{ARS19} S. Artamonov, K. Runovski, H.-J. Schmeisser, Approximation by bandlimited functions, generalized $K$-functionals and generalized moduli of smoothness, Anal. Math. \textbf{45} (2019), 1--24.

\bibitem{BL} E.~Belinsky, E.~Liflyand,  Approximation properties in $L_p$, $0<p<1$,
Functiones et Approximatio, \textbf{XXII} (1993), 189--199.

\bibitem{BDGS77}  P. L. Butzer, H. Dyckhoff, E. G\"{o}rlich, R. L. Stens, Best trigonometric approximation, fractional order derivatives
and Lipschitz classes, Can. J. Math. \textbf{ 29} (1977), 781--793.

\bibitem{DHI} Z.~Ditzian, V.~Hristov, K.~Ivanov, {Moduli of smoothness and
$K$-functional in $L_p$, $0<p<1$}, Constr. Approx. \textbf{11}
(1995), 67--83.

\bibitem{Di98} Z. Ditzian, Fractional derivatives and best approximation,
Acta Math. Hungar. \textbf{81} (1998), no.~4, 323--348.

\bibitem{Dr10} B. R. Draganov,  Exact estimates of the rate of approximation of convolution operators, J. Approx. Theory \textbf{162} (2010), no.~5, 952--979.

\bibitem{johnen}
H. Johnen, K. Scherer, On the equivalence of the $K$-functional and moduli of
continuity and some applications, Constructive theory of functions of several variables
(Proc. Conf., Math. Res. Inst., Oberwolfach 1976), Lecture Notes in Math., vol. 571,
Springer-Verlag, Berlin–Heidelberg 1977,  119--140.

\bibitem{Ho} L. H\"ormander, The analysis of Linear Partial Differential Operators I, Second edition, Springer-Verlag, 1990.

\bibitem{K11} Yu. S. Kolomoitsev, On moduli of smoothness and $K$-functionals of fractional order in the Hardy spaces (Russian) Ukr. Mat. Visn.~\textbf{8} (2011), no. 3, 421--446; translation in J. Math. Sci. \textbf{181} (2012), no. 1, 78--79.

\bibitem{K12}  Yu. S. Kolomoitsev, Approximation properties of generalized Bochner-Riesz means in the Hardy spaces $H_p$, $0<p\le 1$ (Russian) Mat. Sb. \textbf{203}:8 (2012), 79--96; translation in Sb. Math. \textbf{203}:8 (2012), 1151--1168.

\bibitem{K17} Yu. Kolomoitsev, On moduli of smoothness and averaged differences of fractional order, Fract. Calc. Appl. Anal. \textbf{20} (2017), no.~4, 988--1009.


\bibitem{KL19} Yu. Kolomoitsev, T. Lomako, Inequalities in approximation theory involving fractional smoothness in $L_p$, $0 < p < 1$. In: Abell M., Iacob E., Stokolos A., Taylor S., Tikhonov S., Zhu J. (eds) Topics in Classical and Modern Analysis. Applied and Numerical Harmonic Analysis. Birkh\"auser, Cham, 2019.

\bibitem{KT20} Yu. Kolomoitsev, S. Tikhonov, Properties of moduli of smoothness in $L_p(\R^d)$, J. Approx. Theory \textbf{257}, (2020), 105423.

\bibitem{KT20_2} Yu. Kolomoitsev, S. Tikhonov, Smoothness of functions vs. smoothness of approximation processes, Bull. Math. Sci. \textbf{10} (2020), no.~3, 2030002.

\bibitem{KT19m} Yu. Kolomoitsev, S. Tikhonov, Hardy-Littlewood and Ulyanov inequalities, Mem. Amer. Math. Soc. \textbf{271}, no.~1325, 2021.

\bibitem{Lu95} S. Z. Lu,   Four Lectures on Real $H_p$ Spaces, World Scientific Publishing Co., Inc., River Edge, NJ, 1995.


\bibitem{LMN} {D. S. Lubinsky, A. Mate, P. Nevai}, Quadrature sums involving $p$th powers of polynomials, {SIAM J. Math. Anal.} \textbf{18} (1987), 53--544.

\bibitem{nikol-book}
S. M. Nikol'skii, Approximation of Functions of Several Variables and Embedding
Theorems, Springer, New York, 1975.

\bibitem{run}  K. Runovski, Approximation of families of linear polynomial
operators, Disser. of Doctor of Science, Moscow State University,  2010.

\bibitem{RS}  K. Runovski, H.-J. Schmeisser,  On some
extensions of Berenstein's inequality for trigonometric polynomials, Funct. Approx. Comment. Math. {\bf 29} (2001), 125--142.

\bibitem{RS3} K.  Runovski, H.-J. Schmeisser,
General moduli of smoothness and approximation by families of linear polynomial operators,
New Perspectives on Approximation and Sampling Theory,
Applied and Numerical Harmonic Analysis, 2014,  269--298.


\bibitem{SS00} H.-J. Schmeisser, W. Sickel, Sampling theory and function spaces, in: Applied Mathematics Reviews, Vol. 1, World Scientific, 2000, pp. 205--284.


\bibitem{stens} R. L. Stens, Sampling with generalized kernels, in: J.R. Higgins, R.L. Stens (Eds.), Sampling Theory in Fourier and Signal Analysis: Advanced Topics, Clarendon Press, Oxford, 1999.

\bibitem{TribF}
H. Triebel, Theory of Function Spaces,  Basel: Birkh\"{a}user, 2010, Reprint of the 1983 original.

\bibitem{TB} R. M. Trigub, E. S. Belinsky, Fourier Analysis and
Appoximation of Functions, Kluwer, 2004.


\bibitem{Wil} G. Wilmes,  On Riesz-type inequalities and $K$-functionals related to Riesz potentials in $\R^N$,
Numer. Funct. Anal. Optim. {\bf 1} (1) (1979), 57--77.

\bibitem{Z}
A. Zygmund, Trigonometric Series,  Cambridge University Press, 1968.

}

\end{thebibliography}
\end{document}